\documentclass[11pt]{amsart}

\usepackage{lmodern}
\usepackage{microtype}
\usepackage[english]{babel}
\usepackage{mathtools}
\usepackage{hyperref}
\usepackage[hyperpageref]{backref}
\usepackage[msc-links]{amsrefs}
\usepackage{mathrsfs}

\theoremstyle{plain} 
\newtheorem{theorem}{Theorem}
\newtheorem{lemma}[theorem]{Lemma}
\newtheorem{corollary}[theorem]{Corollary}

\theoremstyle{definition} 
\newtheorem*{problem}{Problem}

\DeclareMathOperator{\mre}{Re} 

\newcommand{\sigmaa}{\sigma_{\mathrm{a}}}
\newcommand{\sigmac}{\sigma_{\mathrm{c}}}
\newcommand{\sigmam}{\sigma_{\mathrm{m}}}

\begin{document} 
\title[Mean values, order, and convergence in $\mathscr{H}^2$]{Mean values, order, and convergence in the Hardy space of Dirichlet series}

\begin{abstract}
	Let $\mathscr{H}^2$ denote the Hilbert space of Dirichlet series with square-summable coefficients. It follows from a theorem of Bohr that if $f$ in $\mathscr{H}^2$ has a bounded analytic continuation to the right half-plane, then its norm can be computed from the mean values
	\[\|f\|_{\mathscr{H}^2}^2 = \lim_{\sigma\to 0^+} \lim_{T\to\infty} \frac{1}{2T} \int_{-T}^T \lvert f(\sigma+it) \rvert^2\,dt.\]
	We investigate to what extent the assumption that the analytic continuation be bounded can be relaxed, and obtain optimal results in two different directions. The first direction concerns the order of the analytic continuation, and yields as a corollary a version of the classical Landau--Schnee convergence theorem for $\mathscr{H}^2$. The second direction concerns the abscissa of convergence, and resolves a problem posed by McCarthy.
\end{abstract}

\date{\today}


\thanks{Brevig is supported by Grant 354537 of the Research Council of Norway. Kouroupis is supported by Grant 1203126N of the Research Foundation -- Flanders (FWO)}

\author{Ole Fredrik Brevig} 
\address{Department of Mathematical Sciences, Norwegian University of Science and Technology (NTNU), 7491 Trondheim, Norway} 
\email{ole.brevig@ntnu.no}

\author{Athanasios Kouroupis} 
\address{Department of Mathematics, KU Leuven, Celestijnenlaan 200B, 3001, Leuven, Belgium} 
\email{athanasios.kouroupis@kuleuven.be}

\maketitle

\section{Introduction}
Let $\mathscr{H}^2$ denote the Hilbert space of Dirichlet series with square-summable coefficients. It is well known that $\mathscr{H}^2$ consists of analytic functions in the half-plane $\mathbb{C}_{\frac{1}{2}}$, where $\mathbb{C}_\kappa = \{s = \sigma+it\,:\, \sigma>\kappa\}$, and that there are elements in $\mathscr{H}^2$ that do not admit analytic continuations to any larger domain (see e.g.~\cite{QQ2020}*{Corollary~8.4.1}).

If the Dirichlet series $f(s) = \sum_{n\geq1} a_n n^{-s}$ converges uniformly in $\mathbb{C}_\kappa$, then a direct calculation yields the mean value formula
\begin{equation} \label{eq:mvf}
    \lim_{T\to\infty} \frac{1}{2T} \int_{-T}^T \lvert f(\sigma+it) \rvert^2 \,dt = \sum_{n=1}^\infty \lvert a_n \rvert^2 n^{-2\sigma}
\end{equation}
for every $\sigma>\kappa$. Bohr's theorem~\cite{Bohr1913} asserts that if a somewhere convergent Dirichlet series has a bounded analytic continuation to $\mathbb{C}_\kappa$, then it converges uniformly in $\mathbb{C}_{\kappa+\varepsilon}$ for every $\varepsilon>0$. 

It follows that if an element $f$ of $\mathscr{H}^2$ has a bounded analytic continuation to the right half-plane $\mathbb{C}_0$, then it enjoys the norm formula
\begin{equation} \label{eq:H2norm}
    \|f\|_{\mathscr{H}^2}^2 = \lim_{\sigma \to 0^+} \lim_{T\to\infty} \frac{1}{2T} \int_{-T}^T \lvert f(\sigma+it)\rvert^2\,dt.
\end{equation}
The main motivation behind the present paper is to investigate to what extent the assumption that the analytic continuation be bounded in $\mathbb{C}_0$ may be replaced by a weaker assumption which still ensures that \eqref{eq:H2norm} holds.

We begin with the Lindel\"of order function, which is defined for an analytic function of finite order in $\mathbb{C}_\kappa$ by
\[\mu_f(\sigma) = \limsup_{\lvert t \rvert\to\infty} \frac{\log\lvert f(\sigma+it) \rvert}{\log \lvert t \rvert}\]
for $\sigma>\kappa$. We shall frequently use that $\mu_f$ is convex on $(\kappa,\infty)$. Let $\sigmac(f)$ and $\sigmaa(f)$ denote, respectively, the abscissa of convergence and the abscissa of absolute convergence for the Dirichlet series $f$. If $f \not \equiv 0$, then $\mu_f(\sigma)=0$ for $\sigma>\sigmaa(f)$ and  $\mu_f(\sigma) \leq 1$ for $\sigma>\sigmac(f)$. From the former and convexity, it follows that $\mu_f$ is decreasing on $(\kappa,\infty)$ if $f$ has an analytic continuation of finite order to $\mathbb{C}_\kappa$.

It will be convenient to consider another abscissa that was introduced by Titchmarsh~\cite{Titchmarsh1958}*{\S9.51}. For a somewhere convergent Dirichlet series $f$, let $\sigmam(f)$ denote the infimum of those $\kappa$ such that $f$ has an analytic continuation of finite order to $\mathbb{C}_\kappa$, and such that the mean value formula \eqref{eq:mvf} holds for every $\sigma>\kappa$, in the sense that both sides are finite and equal.

If $\sigmam(f) \leq 0$ for $f$ in $\mathscr{H}^2$, then \eqref{eq:H2norm} follows from \eqref{eq:mvf} by the monotone convergence theorem as $\sigma\to0^+$. Our question may therefore be formulated succinctly as follows: Under what assumptions on a Dirichlet series $f$ in $\mathscr{H}^2$ with analytic continuation to $\mathbb{C}_0$ does it hold that $\sigmam(f) \leq 0$? The two arcs of this paper arise from two different approaches to this question: the first via order and the second via convergence.

It is not difficult to establish that if $\sigma>\sigmam(f)$, then $\mu_f(\sigma) \leq \frac{1}{2}$ (see e.g.~\cite{Titchmarsh1958}*{\S9.55}). In particular, we find that $\mu_f(\sigma) \leq \frac{1}{2}$ for $\sigma>0$ is a necessary condition for the norm formula \eqref{eq:H2norm}. In the first arc, we will investigate what we can say about $\sigmam(f)$ if $\frac{1}{2}$ is replaced by some $\beta>0$.

\begin{theorem} \label{thm:betayes}
    Fix $\beta>0$. If an element $f$ of $\mathscr{H}^2$ has an analytic continuation to $\mathbb{C}_0$ such that $\mu_f(\sigma) \leq \beta$ for every $\sigma>0$, then 
    \[\sigmam(f) \leq \frac{\beta}{1+2\beta}.\]
\end{theorem}

The proof of Theorem~\ref{thm:betayes} uses the Cahen--Mellin formula, and the assumption that $f$ belongs to $\mathscr{H}^2$ enters via the Montgomery--Vaughan inequality.

Recall that an analytic function $f$ is said to have order $0$ in $\mathbb{C}_\kappa$ if $\mu_f(\sigma) \leq 0$ for every $\sigma>\kappa$. We mentioned above that the norm formula \eqref{eq:H2norm} holds if $f$ in $\mathscr{H}^2$ has a bounded analytic continuation to $\mathbb{C}_0$. The following consequence of Theorem~\ref{thm:betayes} relaxes that assumption.

\begin{corollary} \label{cor:zeroorder}
    If $f$ in $\mathscr{H}^2$ has an analytic continuation to $\mathbb{C}_0$ of order $0$, then 
    \[\|f\|_{\mathscr{H}^2}^2 = \lim_{\sigma \to 0^+} \lim_{T\to\infty} \frac{1}{2T} \int_{-T}^T \lvert f(\sigma+it)\rvert^2\,dt.\]
\end{corollary}

To explore further consequences of Theorem~\ref{thm:betayes}, we recall that the three abscissae satisfy
\begin{equation} \label{eq:titchmarsh}
    \sigmac(f) \leq \sigmam(f) \leq \sigmaa(f).
\end{equation}
The first inequality in this chain is due to Titchmarsh~\cite{Titchmarsh1958}*{\S9.55}, while the second is trivial. Fix $\beta>0$. A classical convergence criterion due to Landau and Schnee~\cite{Landau1953}*{Satz~54 in \S238} implies that if $f(s)=\sum_{n\geq1} a_n n^{-s}$ with $a_n = O(n^{\varepsilon})$ for every $\varepsilon>0$ has an analytic continuation to $\mathbb{C}_0$ such that $\mu_f(\sigma) \leq \beta$ for $\sigma>0$, then 
\begin{equation} \label{eq:slbeta}
    \sigmac(f) \leq \frac{\beta}{1+\beta}.
\end{equation}
Strengthening the assumption to square-summable coefficients, we obtain the following version of the Landau--Schnee theorem from \eqref{eq:titchmarsh} and Theorem~\ref{thm:betayes}.

\begin{corollary}[Landau--Schnee in $\mathscr{H}^2$] \label{cor:schneelandau}
    Fix $\beta>0$. If $f$ in $\mathscr{H}^2$ has an analytic continuation to $\mathbb{C}_0$ with $\mu_f(\sigma) \leq \beta$ for $\sigma>0$, then 
    \[\sigmac(f) \leq \frac{\beta}{1+2\beta}.\]
\end{corollary}

We now turn to the question of whether $\frac{\beta}{1+2\beta}$ can be improved in Theorem~\ref{thm:betayes} and Corollary~\ref{cor:schneelandau}. Here we take our inspiration from Bohr~\cite{Bohr1949}*{\S2}, who constructed an entire function that can be represented by a Dirichlet series $f(s) = \sum_{n\geq1} a_n n^{-s}$ with $a_n = O(n^{\varepsilon})$ for every $\varepsilon>0$ satisfying
\[\sigmac(f) = \sigmaa(f) = \frac{\beta}{1+\beta}\]
and whose order function is $\mu_f(\sigma) = \beta-(1+\beta)\sigma$ for $\sigma \leq \frac{\beta}{1+\beta}$. In particular, this demonstrates the optimality of \eqref{eq:slbeta}. Tweaking Bohr's construction to ensure that the coefficients are square-summable has the following effect.

\begin{theorem} \label{thm:betano}
    Fix $\beta>0$. There is an entire function $f$ that belongs to $\mathscr{H}^2$ with
    \[\sigmac(f) = \sigmaa(f) = \frac{\beta}{1+2\beta}\]
    and such that $\mu_f(\sigma) = \beta-(1+2\beta)\sigma$ for $\sigma \leq \frac{\beta}{1+2\beta}$.
\end{theorem}

Theorem~\ref{thm:betano} demonstrates that Corollary~\ref{cor:schneelandau} cannot be improved. It also yields the optimality of Theorem~\ref{thm:betayes} and Corollary~\ref{cor:zeroorder} via \eqref{eq:titchmarsh}, and this concludes the first arc of our paper.

For the second arc of our paper, we return to the norm formula \eqref{eq:H2norm} for some $f$ in $\mathscr{H}^2$. If it holds, then plainly $\sigmac(f) \leq 0$ by \eqref{eq:titchmarsh}. We now wish to understand what happens if the condition $\sigmac(f) \leq 0$ is replaced by $\sigmac(f) \leq \theta$ for some $\theta$. A result due to Landau~\cite{Landau1953}*{Satz~41 in \S226} (see also~\cite{BK2024}*{Section~3}) asserts that
\begin{equation} \label{eq:landau}
    \sigmam(f) \leq \frac{\sigmac(f)+\sigmaa(f)}{2}.
\end{equation}
Since $\sigmaa(f) \leq \frac{1}{2}$ for every $f$ in $\mathscr{H}^2$, we obtain the following counterpart to Corollary~\ref{cor:zeroorder} in the convergence regime from Landau's result. 

\begin{corollary}[Landau] \label{cor:tylandau}
    If $f$ in $\mathscr{H}^2$ has $\sigmac(f)\leq -\frac{1}{2}$, then 
    \[\|f\|_{\mathscr{H}^2}^2 = \lim_{\sigma \to 0^+} \lim_{T\to\infty} \frac{1}{2T} \int_{-T}^T \lvert f(\sigma+it)\rvert^2\,dt.\]
\end{corollary}

Notice here the discrepancy between the necessary $\sigmac(f) \leq 0$ and sufficient $\sigmac(f) \leq -\frac{1}{2}$ conditions for the validity of the norm formula \eqref{eq:H2norm}. As in the first arc, we wish to determine the largest $\theta$ such that $\sigmac(f) \leq \theta$ for $f$ in $\mathscr{H}^2$ ensures that the norm formula \eqref{eq:H2norm} holds.

The key ingredient in the proof of Landau's result \eqref{eq:landau} (see~\cite{Landau1953}*{Satz~35 in \S224} or~\cite{BK2024}*{Section~3}) is that if $f(s) = \sum_{n\geq1} a_n n^{-s}$ and $\sigma>\sigmac(f)$, then we have access to the coefficients via the formula
\[a_n n^{-\sigma} = \lim_{T\to\infty} \frac{1}{2T} \int_{-T}^T f(\sigma+it) \, n^{it}\,dt.\]
It follows from this and a straightforward argument that
\begin{equation} \label{eq:liminf}
    \sum_{n=1}^\infty \lvert a_n \rvert^2 n^{-2\sigma} \leq \liminf_{T\to\infty} \frac{1}{2T} \int_{-T}^T \lvert f(\sigma+it) \rvert^2 \,dt.
\end{equation}
From \eqref{eq:liminf} it is easy to find an example demonstrating that \eqref{eq:landau} is sharp in general. Indeed, the alternating zeta function
\begin{equation}\label{eq:eta}
\eta(s) = \sum_{n=1}^\infty (-1)^{n-1} n^{-s}=(1-2^{1-s})\zeta(s)
\end{equation}
has $\sigmac(\eta)=0$, $\sigmaa(\eta)=1$, and must have $\sigmam(\eta) \geq \frac{1}{2}$ since the left-hand side of \eqref{eq:liminf} in this case is infinite for $\sigma\leq\frac{1}{2}$.

If $f$ is in $\mathscr{H}^2$ and the necessary condition $\sigmac(f)\leq 0$ for the norm formula \eqref{eq:H2norm} holds, then we get $\sigmam(f) \leq \frac{1}{4}$ from Landau's theorem. However, we cannot infer that $\frac{1}{4}$ is sharp via \eqref{eq:liminf}, since the assumption that $f$ is in $\mathscr{H}^2$ means that the left-hand side is finite for every $\sigma>0$. This may have motivated the following problem posed by McCarthy~\cite{JOHN}*{Problem 1}.

\begin{problem}[McCarthy]
    If $f$ is in $\mathscr{H}^2$ and $\sigmac(f) \leq 0$, is it true that $\sigmam(f) \leq 0$?
\end{problem}

Again we take our inspiration from Bohr~\cite{Bohr1950}*{\S1}, who constructed a Dirichlet series $f$ with $\sigmac(f)=0$, $\sigmaa(f)=1$, and
\[\mu_f(\sigma) = 1-\sigma\]
for $0 < \sigma \leq 1$. In view of the facts that $\mu_f(\sigma) \leq 1$ for $\sigma>\sigmac(f)$ and $\mu_f(\sigma)=0$ for $\sigma>\sigmaa(f)$, and of the convexity of the order function discussed above, this means that $\mu_f(\sigma)$ is as large as it can possibly be on this interval. Bohr's construction also demonstrates the optimality of \eqref{eq:landau} due to the bound $\mu_f(\sigma) \leq \frac{1}{2}$ for $\sigma>\sigmam(f)$ discussed above. 

Bohr's example can be shifted and tweaked to obtain a Dirichlet series in $\mathscr{H}^2$, but this yields $\sigmac(f)=-\frac{1}{2}$ and provides no information about whether Corollary~\ref{cor:tylandau} can be improved or about McCarthy's problem. A more elaborate construction yields the following result.

\begin{theorem} \label{thm:theta}
    Fix $-\frac{1}{2}<\theta<\frac{1}{2}$. There is a Dirichlet series $g_\theta$ in $\mathscr{H}^2$ with
    \[\sigmac(g_\theta)=\theta, \qquad \sigmaa(g_\theta)=\frac{1}{2}, \qquad \text{and} \qquad \mu_{g_\theta}(\sigma) = \frac{1-2\sigma}{1-2\theta}\]
    for $\theta < \sigma \leq \frac{1}{2}$.
\end{theorem}

As with Bohr's construction, the order function of $g_\theta$ is as large as possible on the whole interval $\theta < \sigma \leq \frac{1}{2}$. In this sense, Theorem~\ref{thm:theta} may be thought of as a refinement of Bohr's construction. 

Landau's theorem \eqref{eq:landau} and the bound $\mu_f(\sigma) \leq \frac{1}{2}$ for $\sigma>\sigmam(f)$ yield that
\[\sigmam(g_\theta) = \frac{1+2\theta}{4}\]
for the Dirichlet series $g_\theta$ from Theorem~\ref{thm:theta}. For $-\frac{1}{2}<\theta\leq0$, this demonstrates the optimality of Corollary~\ref{cor:tylandau} and resolves McCarthy's problem in the negative.

It is worth mentioning that if the answer to McCarthy's problem were positive, then it would have significant consequences for the order of the Riemann zeta function.

Let $\Delta_k(x)$ denote the error term in the generalized divisor problem, so that
\[\sum_{n\leq x} d_k(n) = xP_{k-1}(\log{x}) + \Delta_k(x)\]
for a polynomial $P_{k-1}$ of degree $k-1$, and let $\alpha_k$ denote the infimum of the numbers $\alpha$ for which $\Delta_k(x) = O\left(x^{\alpha+\varepsilon}\right)$ holds for every $\varepsilon>0$. Since the abscissa of convergence of a Dirichlet series is determined by the growth of the summatory function of its coefficients (see e.g.~\cite{Titchmarsh1958}*{\S9.14}), it follows from \eqref{eq:eta} and a calculation that
\[\sigmac(\eta^k) \leq \alpha_k.\]
Set $\sigma_k = \max(\alpha_k,\frac{1}{2})$ and fix $\delta>0$. The function $f(s) = \eta^k(\sigma_k+\delta+s)$ is then in $\mathscr{H}^2$ and satisfies $\sigmac(f)<0$, so a positive answer to McCarthy's problem would give $\sigmam(\eta^k)\leq\sigma_k$ and, consequently, that
\[\mu_\zeta(\sigma_k) = \mu_\eta(\sigma_k) \leq \frac{1}{2k}.\]
Among the known estimates for $\alpha_k$, the most striking bound that could be obtained in this way comes from $k=4$, where $\alpha_4\leq\frac{1}{2}$ by a result of Tong~\cite{Tong1953}. This is the largest $k$ for which the shift still reaches the critical line, and it gives
\[\mu_\zeta\left(\frac12\right)\leq\frac18.\]
Note that the best known bound for the order of the zeta function on the critical line is due to Bourgain~\cite{Bourgain2017}, who proved that $\mu_\zeta\left(\frac{1}{2}\right)\leq\frac{13}{84}$.

It was this connection that attracted us to McCarthy's problem and suggested that we should try to construct the example in Theorem~\ref{thm:theta}.

\subsection*{AI disclosure}
The mathematical ideas and constructions that led to the three main results (Theorem~\ref{thm:betayes}, Theorem~\ref{thm:betano}, and Theorem~\ref{thm:theta}) are due to the authors. The authors used Claude Opus 5/5.5 and GPT-5.6 Sol to distill the arguments, polish the presentation, and search and translate the literature. The final version of the paper was written by the authors, who take full responsibility for its content.

\subsection*{Organization} This paper contains three additional sections. They are devoted, respectively, to the proofs of Theorem~\ref{thm:betayes}, Theorem~\ref{thm:betano}, and Theorem~\ref{thm:theta}. The material can be read in any order, but the proof of Theorem~\ref{thm:betano} uses Theorem~\ref{thm:betayes} via Corollary~\ref{cor:schneelandau}.

\section{Proof of Theorem~\ref{thm:betayes}}
As mentioned in the introduction, the proof relies on two classical results. The first is the Cahen--Mellin formula (see~\cite{Titchmarsh1958}*{\S9.43}).

\begin{lemma} \label{lem:cm}
	Let $f(s)=\sum_{n\geq1}a_n n^{-s}$ be a somewhere convergent Dirichlet series and
	fix $\kappa>\max(0,\sigmaa(f))$. If $\delta>0$, then
	\[f_\delta(s) = \sum_{n=1}^\infty a_n e^{-n\delta}\,n^{-s}
	= \frac{1}{2\pi i}\int_{\kappa-i\infty}^{\kappa+i\infty}
	\Gamma(z)\,f(s+z)\,\delta^{-z}\,dz\]
	for every $s$ in $\mathbb{C}_0$.
\end{lemma} 

The second is the Montgomery--Vaughan inequality (see \cite{MV1974}*{Corollary~2}).

\begin{lemma} \label{lem:mv}
	If $\sum_{n\geq1} n\lvert b_n \rvert^2<\infty$ and $T>0$, then
	\[\left\lvert\frac{1}{2T}\int_{-T}^{T}\left\lvert \sum_{n=1}^\infty b_n n^{-it}\right\rvert^2\,dt
	-\sum_{n=1}^\infty\lvert b_n\rvert^2\right\rvert
	\leq \frac{9 \pi}{4T}\sum_{n=1}^\infty n \lvert b_n\rvert^2.\]
\end{lemma}

We are now ready to proceed with the proof of Theorem~\ref{thm:betayes}. 

\begin{proof}[Proof of Theorem~\ref{thm:betayes}]
    For $T>0$, we take $\delta=T^{-\alpha}$, where $\alpha>0$ is a parameter that will be chosen later. The goal is to show that
    \begin{equation} \label{eq:limT}
        \lim_{T\to\infty} \left\lvert \left(\frac{1}{2T}\int_{-T}^T \lvert f(\sigma+it)\rvert^2 \,dt\right)^{\frac{1}{2}} -\left(\sum_{n=1}^\infty \lvert a_n \rvert^2 e^{-2n\delta} n^{-2\sigma}\right)^{\frac{1}{2}}\right\rvert=0
    \end{equation}
    for $\frac{\beta}{1+2\beta}<\sigma \leq \frac{1}{2}$. The idea is to approximate the Dirichlet series $f$ using the absolutely convergent series $f_\delta$ from Lemma~\ref{lem:cm} with $\kappa=1$. To this end, we apply the triangle inequality in \eqref{eq:limT} after adding and subtracting the quantity
    \[\left(\frac{1}{2T}\int_{-T}^T \lvert f_\delta(\sigma+it)\rvert^2 \,dt \right)^{\frac{1}{2}}.\]
    To show that the two terms obtained from the triangle inequality go to $0$ as $T\to\infty$, we will use the order assumption via Lemma~\ref{lem:cm} and the $\mathscr{H}^2$ assumption via Lemma~\ref{lem:mv}, respectively. These two estimates place opposing requirements on the parameter $\alpha$. It is possible to choose an $\alpha$ that satisfies both requirements precisely when $\sigma > \frac{\beta}{1+2\beta}$.
    
    To handle the first term, we require two estimates.
    The first follows from the assumption that $\mu_f(\sigma) \leq \beta$ for every $\sigma>0$ and the Phragm\'{e}n--Lindel\"{o}f principle. For every $\varepsilon>0$, there is a constant $B_\varepsilon>0$ such that if $s=\sigma+it$ and $\sigma\geq\varepsilon$, then 
    \begin{equation} \label{eq:PL}
        \lvert f(s) \rvert \leq B_\varepsilon (1+\lvert t \rvert)^{\beta}.
    \end{equation}
    The second is a coarse version of Stirling's formula (see e.g.~\cite{Titchmarsh1958}*{\S4.42}). There is an absolute constant $A>0$ such that if $z=x+iy$ and $-\frac{1}{2} \leq x \leq 1$, then 
    \begin{equation} \label{eq:stirling}
        \lvert \Gamma(z) \rvert \leq A \frac{e^{-\lvert y \rvert}}{\lvert z \rvert}.
    \end{equation}
    For $0<\varepsilon < \sigma \leq \frac{1}{2}$, we use \eqref{eq:PL} and \eqref{eq:stirling} on the horizontal segments to shift the contour in the definition of $f_\delta$ from $\mre{z}=\kappa$ to $\mre{z}=\varepsilon-\sigma$, thereby picking up the simple pole of $\Gamma$ at the origin. This yields
    \[f_\delta(s) - f(s) = \frac{1}{2\pi i} \int_{\varepsilon-\sigma-i\infty}^{\varepsilon-\sigma+i\infty}\Gamma(z)\,f(s+z)\,\delta^{-z}\,dz.\]
    Using \eqref{eq:PL} and \eqref{eq:stirling} again, we infer from this that
    \[\lvert f_\delta(s) - f(s)\rvert \leq A B_\varepsilon C_\beta \frac{(1+\lvert t \rvert)^\beta \delta^{\sigma-\varepsilon}}{\sigma-\varepsilon} \leq A B_\varepsilon C_\beta 2^\beta \frac{T^{\beta-\alpha(\sigma-\varepsilon)}}{\sigma-\varepsilon},\]
    where $C_\beta = \frac{1}{2\pi} \int_{-\infty}^\infty (1+\lvert y \rvert)^\beta e^{-\lvert y \rvert}\,dy$, and the second inequality holds for $\lvert t \rvert \leq T$ when $T\geq1$. Hence if
    \begin{equation} \label{eq:alpha1}
        \alpha > \frac{\beta}{\sigma-\varepsilon},
    \end{equation}
    then    
    \begin{equation} \label{eq:term1}
        \lim_{T\to\infty} \left\lvert \left(\frac{1}{2T}\int_{-T}^T \lvert f(\sigma+it)\rvert^2 \,dt\right)^{\frac{1}{2}} -\left(\frac{1}{2T}\int_{-T}^T \lvert f_\delta(\sigma+it)\rvert^2 \,dt\right)^{\frac{1}{2}}\right\rvert=0.
    \end{equation}

    For the second term, we use that $\lvert a-b \rvert^2 \leq \lvert a^2-b^2 \rvert$ for $a,b\geq0$ and Lemma~\ref{lem:mv} to deduce
    \begin{multline*}
        \left\lvert \left(\frac{1}{2T}\int_{-T}^T \lvert f_\delta(\sigma+it)\rvert^2 \,dt\right)^{\frac{1}{2}} -\left(\sum_{n=1}^\infty \lvert a_n \rvert^2 e^{-2n\delta} n^{-2\sigma}\right)^{\frac{1}{2}}\right\rvert \\ \leq \left(\frac{9}{T} \sum_{n=1}^\infty n^{1-2\sigma} \lvert a_n \rvert^2 e^{-2n\delta}\right)^{\frac{1}{2}}.
    \end{multline*}
    If $0 \leq \sigma \leq \frac{1}{2}$ and $\delta>0$, then $T^{-1} n^{1-2\sigma} e^{-2n\delta} \leq T^{-1}\delta^{-(1-2\sigma)} = T^{\alpha(1-2\sigma)-1}$. Under the proviso that 
    \begin{equation} \label{eq:alpha2}
        (1-2\sigma)\alpha < 1,
    \end{equation}
    it follows from this and the assumption that $f$ is in $\mathscr{H}^2$ that
    \begin{equation} \label{eq:term2}
        \lim_{T\to\infty} \left\lvert \left(\frac{1}{2T}\int_{-T}^T \lvert f_\delta(\sigma+it)\rvert^2 \,dt\right)^{\frac{1}{2}} -\left(\sum_{n=1}^\infty \lvert a_n \rvert^2 e^{-2n\delta} n^{-2\sigma}\right)^{\frac{1}{2}}\right\rvert =0.
    \end{equation}
    We get the desired conclusion \eqref{eq:limT} from \eqref{eq:term1} and \eqref{eq:term2}, provided we can find $\alpha>0$ that satisfies both \eqref{eq:alpha1} and \eqref{eq:alpha2}. By simple algebra, this is possible if and only if we can find $0<\varepsilon<\frac{1}{2}$ such that
    \[\frac{\beta}{1+2\beta} + \frac{\varepsilon}{1+2\beta} < \sigma. \qedhere\]
\end{proof}

\section{Proof of Theorem~\ref{thm:betano}}
We begin by recalling the setup used by Bohr~\cite{Bohr1949}. The finite difference operator of order $r\geq0$ and span $d\geq1$ acts on the Dirichlet monomials $\{n^{-s}\}_{n\geq1}$ by
\begin{equation} \label{eq:Ddr}
    \Delta_d^r(n^{-s}) = \sum_{j=0}^r (-1)^j \binom{r}{j} (n+jd)^{-s}.
\end{equation}
The formula
\begin{multline*}
    \Delta_d^r(n^{-s}) = s(s+1)\cdots(s+r-1) \\ \times \int_n^{n+d}\int_{x_1}^{x_1+d} \cdots\int_{x_{r-1}}^{x_{r-1}+d} x_r^{-s-r}\,dx_r\,dx_{r-1} \cdots\,dx_1,
\end{multline*}
roughly speaking, makes it possible to trade $n^{-\sigma}$ for $n^{-\sigma-r}$ at the cost of $(\lvert s \rvert d)^r$ when $\sigma+r>0$. We will use the following more flexible (and accurate) version of this estimate, which is due to Bohr~\cite{Bohr1949}*{p.~15}.

\begin{lemma} \label{lem:rsplit}
    Let $r = r_1+r_2$ for $r_1,r_2\geq0$ and $d\geq1$. If $\sigma+r_2 >0$, then
    \[\lvert \Delta_d^r(n^{-s}) \rvert \leq 2^{r_1} \left(\prod_{k=0}^{r_2-1}\lvert s+k\rvert\right) d^{r_2} n^{-\sigma-r_2}.\]
\end{lemma}

\begin{proof}
    If $r=0$ there is nothing to do, since $\Delta_d^0(n^{-s}) = n^{-s}$.

    Suppose first that $r_1=0$ so that $r_2=r$. The idea is simply to move absolute values into the integral representation for $\Delta_d^r$ given above. The domain of integration has measure $d^r$ and since
    \[n \leq x_1 \leq x_2 \leq \cdots \leq x_r,\]
    the integrand satisfies $\lvert x_r^{-s-r} \rvert \leq n^{-\sigma-r}$ by the assumption $\sigma+r>0$.

    If $r_1>0$, then the identity $\Delta_d^r = \Delta_d^{r_1} \Delta_d^{r_2}$ and the triangle inequality yield
    \[\lvert \Delta_d^r(n^{-s}) \rvert \leq \sum_{j=0}^{r_1} \binom{r_1}{j} \lvert \Delta_d^{r_2}((n+jd)^{-s})\rvert.\]
    The stated estimate now follows from the case $r_1=0$ established above, since $(n+jd)^{-\sigma-r_2} \leq n^{-\sigma-r_2}$ by the assumption $\sigma+r_2>0$.
\end{proof}

We now introduce the bricks of Bohr's construction.\footnote{Our parameters relate to those in \cite{Bohr1949}*{\S2} by $N_m=p_m$, $d_m=q_m$, and $\frac{2\beta}{1+2\beta}=\alpha$.} Fix $\beta>0$. For $m=1,2,3,\ldots$ we set
\[N_m = 2^{2^m}, \qquad d_m = \left\lceil N_m^{\frac{2\beta}{1+2\beta}}\right\rceil, \qquad P_m(s) = \sum_{n=N_m}^{N_m+d_m-1} \Delta_{d_m}^m(n^{-s}).\]

The choice of $N_m$ ensures that the bricks $P_m$ are supported on disjoint sets of Dirichlet monomials. The key property of $d_m$ is the exponent $\frac{2\beta}{1+2\beta}$ relating it to $N_m$: this single exponent fixes both the abscissae of the Dirichlet series we construct and the slope of its order function. The heart of the construction is the following pointwise upper bound for the bricks.

\begin{lemma} \label{lem:pbrick}
    If $m \geq r \geq 0$ and $\sigma+r>0$, then
    \[\lvert P_m(s) \rvert \leq 2^{m+1}
    \left(\prod_{k=0}^{r-1} \lvert s+k \rvert\right)
    N_m^{\frac{2\beta-r}{1+2\beta}-\sigma}.\]
\end{lemma}

\begin{proof}
    We use the triangle inequality and then Lemma~\ref{lem:rsplit} with $r_1=m-r$ and $r_2=r$ to infer that
    \[\lvert P_m(s) \rvert \leq d_m 2^{m-r} \left(\prod_{k=0}^{r-1} \lvert s+k \rvert\right) d_m^r N_m^{-\sigma-r},\]
    since $n^{-\sigma-r} \leq N_m^{-\sigma-r}$ for $N_m \leq n < N_m+d_m$ by the assumption that $\sigma+r>0$. The proof is completed by the bound $d_m \leq 2 N_m^{\frac{2\beta}{1+2\beta}}$.
\end{proof}

\begin{lemma} \label{lem:h2brick}
    Fix $m\geq1$. The brick $P_m$ is a Dirichlet polynomial supported on $N_m \leq n < N_{m+1}$ with
    \[\|P_m\|_{\mathscr{H}^2}^2 = d_m \binom{2m}{m}.\]
\end{lemma}

\begin{proof}
    We first note that the $d_m$ finite differences $\Delta_{d_m}^m(n^{-s})$ appearing in the definition of $P_m$ are Dirichlet polynomials supported on the integers
    \[n, n+d_m, n+2d_m,\ldots,n+md_m.\]
    Since $n$ runs from $N_m$ to $N_m+d_m-1$, it follows that the $d_m$ finite differences are supported on pairwise disjoint sets of integers. They are therefore orthogonal, whence
    \[\|P_m\|_{\mathscr{H}^2}^2 = d_m \sum_{j=0}^m \binom{m}{j}^2 = d_m \binom{2m}{m}\]
    by \eqref{eq:Ddr} and Vandermonde's identity.
\end{proof}

For our purposes, it will be enough to know that $\|P_m\|_{\mathscr{H}^2}^2\leq d_m 4^m$. We will normalize our bricks by dividing by $m 2^m \sqrt{d_m}$, where the extra factor $m$ ensures convergence but plays no role in the rest of the argument.

\begin{proof}[Proof of Theorem~\ref{thm:betano}]
    With $\beta>0$ fixed as above, we put
    \begin{equation} \label{eq:fbeta}
        f_\beta(s) = \sum_{m=1}^\infty \frac{P_m(s)}{m 2^m \sqrt{d_m}}.
    \end{equation}
    We have already seen (via Lemma~\ref{lem:h2brick}) that $f_\beta$ is in $\mathscr{H}^2$.

    We continue with the assertions about the abscissae, where it will be convenient to use \eqref{eq:Ddr} and write out the Dirichlet series expansion
    \[f_\beta(s) = \sum_{n=1}^\infty a_n n^{-s} = \sum_{m=1}^\infty \frac{1}{m 2^m \sqrt{d_m}} \sum_{n=N_m}^{N_m+d_m-1} \sum_{j=0}^m (-1)^j \binom{m}{j} (n+jd_m)^{-s}.\]
    For this part of the argument, we will assume that $\sigma\geq0$. As discussed above, each integer of the form $n+jd_m$ occurs only once in this expansion. Estimating by $n+jd_m \geq N_m$ and $d_m \leq 2N_m^{\frac{2\beta}{1+2\beta}}$, we obtain
    \[\sum_{n=1}^\infty \lvert a_n \rvert n^{-\sigma} \leq \sum_{m=1}^\infty \frac{\sqrt{d_m}}{m}N_m^{-\sigma} \leq \sum_{m=1}^\infty \frac{\sqrt{2}}{m} N_m^{\frac{\beta}{1+2\beta}-\sigma}.\]
    This demonstrates that $\sigmaa(f_\beta) \leq \frac{\beta}{1+2\beta}$. The integers $N_m \leq n < N_m+d_m$ are precisely those of the $m$th brick with $j=0$, so their coefficients are all equal and the partial sums satisfy
    \begin{multline*}
        S_{N_m+d_m-1}f_\beta(\sigma)-S_{N_m-1}f_\beta(\sigma)
        \\ = \frac{1}{m 2^m \sqrt{d_m}}\sum_{n=N_m}^{N_m+d_m-1} n^{-\sigma} \geq \frac{2^{-\sigma}}{m 2^m} N_m^{\frac{\beta}{1+2\beta}-\sigma},
    \end{multline*}
    since $N_m^{\frac{2\beta}{1+2\beta}} \leq d_m \leq N_m$. This demonstrates that $\sigmac(f_\beta) \geq \frac{\beta}{1+2\beta}$ and hence that
    \[\sigmac(f_\beta) = \sigmaa(f_\beta) = \frac{\beta}{1+2\beta}.\]

    Let us next establish that $f_\beta$ is entire. Fix $r\geq1$ and decompose \eqref{eq:fbeta} as
    \[f_\beta(s) = \sum_{m=1}^{r-1} \frac{1}{m 2^m \sqrt{d_m}} P_m(s) + \sum_{m=r}^\infty \frac{1}{m 2^m \sqrt{d_m}} P_m(s).\]
    The first sum is a Dirichlet polynomial and hence entire. To estimate the second sum, we assume that $\sigma+r>0$ and use Lemma~\ref{lem:pbrick} and $d_m \geq N_m^{\frac{2\beta}{1+2\beta}}$ to obtain
    \begin{equation} \label{eq:tailest}
        \sum_{m=r}^\infty \frac{1}{m 2^m \sqrt{d_m}} \lvert P_m(s) \rvert \leq \left(\prod_{k=0}^{r-1} \lvert s+k \rvert\right)\sum_{m=r}^\infty \frac{2}{m} N_m^{\frac{\beta-r}{1+2\beta}-\sigma}.
    \end{equation}
    The series on the right-hand side of \eqref{eq:tailest} converges for $\sigma>\frac{\beta-r}{1+2\beta}$. Note that this is a stronger requirement than our assumption $\sigma+r>0$. This demonstrates that \eqref{eq:fbeta} converges locally uniformly in the half-plane $\mathbb{C}_{\frac{\beta-r}{1+2\beta}}$ for every $r=1,2,3,\ldots$, and so $f_\beta$ is entire.

    All that remains is to compute the order function. The bound \eqref{eq:tailest} also yields that $\mu_{f_\beta}(\sigma) \leq r$ for $\sigma > \frac{\beta-r}{1+2\beta}$ and thus, by continuity,
    \begin{equation} \label{eq:mur}
        \mu_{f_\beta}\left(\frac{\beta-r}{1+2\beta}\right) \leq r
    \end{equation}
    for $r=1,2,3,\ldots$. We also have
    \begin{equation} \label{eq:twomore}
        \mu_{f_\beta}\left(\frac{\beta}{1+2\beta}\right)=0 \qquad \text{and} \qquad \mu_{f_\beta}(0)\geq \beta,
    \end{equation}
    the former due to the general fact $\mu_{f_\beta}(\sigmaa(f_\beta))=0$ and the latter from Corollary~\ref{cor:schneelandau} via contradiction: if $\mu_{f_\beta}(0)<\beta$, then $\sigmac(f_\beta) < \frac{\beta}{1+2\beta}$ since $f_\beta$ is in $\mathscr{H}^2$, which contradicts $\sigmac(f_\beta) = \frac{\beta}{1+2\beta}$ established above. The only convex function on $(-\infty,\frac{\beta}{1+2\beta}]$ that satisfies both \eqref{eq:mur} and \eqref{eq:twomore} is
    \[\mu_{f_\beta}(\sigma) = \beta-(1+2\beta)\sigma,\]
    since $0$ is an interior point of the interval $[\frac{\beta-r}{1+2\beta},\frac{\beta}{1+2\beta}]$ when $r>\beta$.
\end{proof}

\section{Proof of Theorem~\ref{thm:theta}}
The idea behind the construction is fairly simple. The Dirichlet series
\[g(s) = \sum_{n=2}^\infty \frac{1}{\sqrt{n}\log{n}} n^{-s}\]
is in $\mathscr{H}^2$ with $\sigmac(g)=\sigmaa(g)=\frac{1}{2}$, which should be thought of as the endpoint $\theta=\frac{1}{2}$ in Theorem~\ref{thm:theta}. We shall take the dyadic decomposition of $g$ and translate its $m$th component vertically, which transfers its mass from $t=0$ to height $t=\tau_m$. After putting the Dirichlet series back together, we need to estimate the contribution of the different components. This will be handled by the following estimate, which follows from a basic result about exponential sums (see e.g.~\cite{Titchmarsh1986}*{Lemma~4.8}) and Abel summation.

\begin{lemma} \label{lem:vdc}
    There is an absolute constant $C$ such that if $2 \leq N \leq M <2N$ and if $\sigma\geq-\frac{1}{2}$ and $0<\lvert t\rvert\leq \frac{N}{2}$, then
    \[\left\lvert \sum_{n=N}^M n^{-\frac{1}{2}-s} \right \rvert \leq C\frac{N^{\frac{1}{2}-\sigma}}{\lvert t \rvert}.\]
\end{lemma}

We now proceed with the proof of Theorem~\ref{thm:theta}. In order to simplify the expressions, we make the cosmetic adjustment of replacing $\log{n}$ by $m$ on the dyadic blocks. 

\begin{proof}[Proof of Theorem~\ref{thm:theta}]
    Fix $-\frac{1}{2}<\theta<\frac{1}{2}$. Let $\tau_m = 2^{(\frac{1}{2}-\theta)m}$ and define
    \[g_\theta(s) = \sum_{m=1}^\infty Q_m(s) \qquad \text{for} \qquad Q_m(s) = \frac{1}{m}\sum_{n=2^m}^{2^{m+1}-1} n^{i\tau_m}\,n^{-\frac{1}{2}-s}.\]
    It is plain that $g_\theta$ is in $\mathscr{H}^2$ and that $\sigmaa(g_\theta)=\frac{1}{2}$. We claim that it is enough to establish the upper bound
    \begin{equation} \label{eq:sigmactheta}
        \sigmac(g_\theta) \leq \theta
    \end{equation}
    and the lower bound
    \begin{equation} \label{eq:mutheta}
        \mu_{g_\theta}\left(\frac{1+2\theta}{4}\right) \geq \frac{1}{2}.
    \end{equation}
    Indeed, since $\mu_{g_\theta}(\sigma) \leq 1$ for $\sigma>\sigmac(g_\theta)$ and $\mu_{g_\theta}(\sigma)=0$ for $\sigma\geq\sigmaa(g_\theta)$, we get by convexity that
    \[\mu_{g_\theta}(\sigma) \leq \frac{\sigmaa(g_\theta)-\sigma}{\sigmaa(g_\theta)-\sigmac(g_\theta)} = \frac{\frac{1}{2}-\sigma}{\frac{1}{2}-\sigmac(g_\theta)} \leq \frac{1-2\sigma}{1-2\theta}\]
    by \eqref{eq:sigmactheta}, for $\theta < \sigma \leq \frac{1}{2}$. Since this upper bound is linear, $\mu_{g_\theta}$ is convex, and \eqref{eq:mutheta} shows that the bound is attained at the interior point $\frac{1+2\theta}{4}$, it follows that it is attained throughout. The same argument also shows that $\sigmac(g_\theta) < \theta$ would contradict \eqref{eq:mutheta}, whence $\sigmac(g_\theta) = \theta$.

    It therefore remains to establish \eqref{eq:sigmactheta} and \eqref{eq:mutheta}. Discarding finitely many blocks affects neither assertion. We may consequently restrict our attention to $m \geq m_0 = \lceil (\frac{1}{2}+\theta)^{-1}\rceil$, which means that
    \[\tau_m = 2^{(\frac{1}{2}-\theta)m} \leq 2^{m-1}.\]
    The point of this estimate is that we may apply Lemma~\ref{lem:vdc} to (parts of) the blocks $Q_m$. Taking $s=\sigma-i\tau_m$ with $\sigma\geq-\frac{1}{2}$ and $N=2^m$, we obtain
    \begin{equation} \label{eq:Qmsigma}
        \left\lvert\frac{1}{m}\sum_{n=2^m}^M n^{i\tau_m} n^{-\frac{1}{2}-\sigma}\right\rvert \leq \frac{C}{m}\,2^{m(\theta-\sigma)}
    \end{equation}
    for $2^m \leq M < 2^{m+1}$. A straightforward argument based on \eqref{eq:Qmsigma} and the triangle inequality shows that the Dirichlet series $g_\theta$ converges for $s=\sigma>\theta$, from which \eqref{eq:sigmactheta} follows by the well-known fact that Dirichlet series converge in half-planes.

    Turning now to \eqref{eq:mutheta}, we write $\sigma_0 = \frac{1+2\theta}{4}$. We shall evaluate $g_\theta$ at $\sigma_0 + i\tau_k$ for $k\geq m_0$. We first use the triangle inequality to extract the main contribution,
    \begin{equation} \label{eq:triangle}
        \lvert g_\theta(\sigma_0+i\tau_k) \rvert \geq Q_k(\sigma_0+i\tau_k) - \sum_{m \neq k} \lvert Q_m(\sigma_0+i\tau_k)\rvert.
    \end{equation}
    We need a lower bound for the first quantity and an upper bound for the second. The first is easy, since
    \begin{equation} \label{eq:lower}
        Q_k(\sigma_0+i\tau_k) = \frac{1}{k}\sum_{n=2^k}^{2^{k+1}-1} n^{-\frac{1}{2}-\sigma_0} \geq \frac{2^{-\frac{3+2\theta}{4}}}{k} \tau_k^{\frac{1}{2}}.
    \end{equation}
    For $m \neq k$ with $\tau_k \leq 2^{m-1}$, we have
    \[(1-2^{\theta-\frac{1}{2}})\max(\tau_m,\tau_k) \leq \lvert \tau_m - \tau_k\rvert \leq \max(\tau_m,\tau_k) \leq 2^{m-1}.\]
    Using Lemma~\ref{lem:vdc} with $N=2^m$, $M=2^{m+1}-1$, and $s=\sigma_0+i(\tau_k-\tau_m)$, as in \eqref{eq:Qmsigma}, we therefore obtain
    \[\lvert Q_m(\sigma_0+i\tau_k)\rvert \leq \frac{C}{1-2^{\theta-\frac{1}{2}}}\,\frac{\sqrt{\tau_m}}{m\,\max(\tau_m,\tau_k)}.\]
    Summing this over $m>k$ and over those $m<k$ with $\tau_k \leq 2^{m-1}$, and using
    \[\sum_{m>k} \frac{1}{\sqrt{\tau_m}} \leq \frac{1}{(2^{\frac{1-2\theta}{4}}-1)\tau_k^{\frac{1}{2}}} \qquad \text{and} \qquad \sum_{m<k} \sqrt{\tau_m} \leq \frac{\sqrt{\tau_k}}{2^{\frac{1-2\theta}{4}}-1},\]
    it follows that
    \begin{equation} \label{eq:upper1}
        \sum_{\substack{m \neq k \\ \tau_k \leq 2^{m-1}}} \lvert Q_m(\sigma_0+i\tau_k)\rvert \leq \frac{2C}{(1-2^{\theta-\frac{1}{2}})(2^{\frac{1-2\theta}{4}}-1)}\,\frac{1}{\sqrt{\tau_k}}.
    \end{equation}
    It remains to consider those $m$ with $\tau_k>2^{m-1}$, which satisfy $m<k$ since $k \geq m_0$ gives $m-1<(\frac{1}{2}-\theta)k\leq k-1$. Here Lemma~\ref{lem:vdc} is unavailable, and we use the triangle inequality instead, which gives
    \begin{equation} \label{eq:upper2}
        \sum_{2^{m-1}<\tau_k} \lvert Q_m(\sigma_0+i\tau_k)\rvert \leq \sum_{2^{m-1}<\tau_k} \frac{\sqrt{\tau_m}}{m} \leq \frac{2^{\frac{1}{2}-\theta}}{2^{\frac{1-2\theta}{4}}-1}\tau_k^{\frac{1-2\theta}{4}}.
    \end{equation}
    Inserting \eqref{eq:lower}, \eqref{eq:upper1}, and \eqref{eq:upper2} into \eqref{eq:triangle} and noting that $\frac{1-2\theta}{4}<\frac{1}{2}$, we can let $k\to \infty$ and obtain \eqref{eq:mutheta}.
\end{proof}

\bibliography{mocdiri}

\end{document}